\documentclass[11pt]{amsart}
\theoremstyle{plain}
\newtheorem{thm}{Theorem}[section]
\newtheorem{theorem}[thm]{Theorem}

\newtheorem{lemma}[thm]{Lemma}
\newtheorem{corollary}[thm]{Corollary}
\newtheorem{proposition}[thm]{Proposition}
\theoremstyle{definition}
\newtheorem{remark}[thm]{Remark}

\newtheorem{question}[thm]{Question}

\numberwithin{equation}{section}

\title [Free auromorphisms of positive entropy]{Free automorphisms of positive entropy on smooth K\"ahler surfaces}

\author{Keiji Oguiso}

\address{Keiji Oguiso, Department of Mathematics, Osaka University\\
Toyonaka 560-0043 Osaka, Japan and  Korea Institute for Advanced Study, Hoegiro 87, Seoul, 130-722, Korea} \email{oguiso@math.sci.osaka-u.ac.jp}


\thanks{supported by JSPS Gran-in-Aid (B) No 22340009, JSPS Grant-in-Aid (S), No 22224001, and by KIAS Scholar Program}

\begin{document}

\maketitle

\begin{abstract}
We prove that there is a projective K3 surface admitting a (fixed point) free automorphism of positive entropy and that no smooth compact K\"ahler surface other than projective K3 surfaces and their blow up admits such an automorphism. 
\end{abstract}

\section{Introduction}

Throughout this note, we work over the complex number field ${\mathbf C}$.

Let $M$ be a smooth compact K\"ahler manifold and $g \in {\rm Aut}\, (M)$ 
be a biholomorphic automorphism of $M$. The maximum $d_1(g)$ 
of absolute values of 
eigenvalues of the ${\mathbf C}$-linear extension of $g^* \vert H^2(M, {\mathbf Z})$ is called the {\it first dynamical degree} of $g$. By the fundamental result of Gromov-Yomdin (\cite{Gr}, \cite{Yo}, see also \cite{Fr}, Theorem (2.1) 
for explicit statement) and Dinh-Sibony 
(\cite{DS}, Corollary (2.2)), $g$ is of {\it positive entropy} (resp. null entropy) if and only if $d_1(g) > 1$ 
(resp. $d_1(g) = 1$). We note that if $g$ is not of null entropy, then $g$ 
is of positive entropy by ${\rm det}\, g^* \vert H^2(M, {\mathbf Z}) = \pm 1$. 
We also note that $\dim\, M \ge 2$ and $g$ is of infinite order if $g$ is of positive entropy. Let
$$M^g := \{x \in M \vert g(x) = x\}\,\, .$$
$g$ is said to be {\it free} if $M^g = \emptyset$. For instance, a non-trivial translation on a complex torus is free but it is of {\it null} entropy.

Both $M^g$ and $d_1(g)$ are the most basic invariants in studying automorphisms of compact K\"ahler manifolds from complex dynamical point of view. For instance, McMullen (\cite{Mc}) has found very impressive K3 surface automorphisms with Siegel disk. Both fixed point set giving candidates of the center of the Siegel disk and the first dynamical degree played crucial roles in his study.

Concerning these two invariants, Professor Shu Kawaguchi asked me the following:

\begin{question}\label{kawaguchi}
Are there smooth projective surfaces with free automorphisms of positive entropy?
\end{question}

The aim of this short note is to give the following answer to the question:

\begin{theorem}\label{main}

(1) Let $S$ be a smooth compact K\"ahler surface with a free automorphism of positive entropy. Then $S$ is birational to a 
projective K3 surface of Picard number $\ge 2$. 

(2) Conversely, there is a projective K3 surface of Picard number $2$ with a 
free automorphism of positive entropy. 

\end{theorem}

We prove (1) in section 2. Our proof is based on a weaker version of the topological Lefschetz fixed point formula, a weaker version of the holomorphic Lefschetz fixed point formula (see eg. \cite{GH}, Chapter 3, section 4) and the following basic result of Cantat (\cite{Ca}, Proposition 1):

\begin{theorem}\label{cantat}
Let $S$ be a smooth compact K\"ahler surface admitting an automorphism of positive entropy. Then $S$ is bimeromorphic to either ${\mathbf P}^2$, $2$-dimensional complex torus, an Enriques surface, or a K3 surface. 
\end{theorem}

It may be also worth comparing (1) with the fact that a K3 surface admitting an automorphism with Siegel disk (hence of positive entropy and with a fixed point) is of algebraic dimension $0$ (\cite{Mc}, see also \cite{Og}, Example (2.5)).

We shall prove (2) by using the global Torelli theorem for K3 surfaces 
(see eg. \cite{BHPV}, Chapter VIII, Theorem (11.1)) and 
the {\it golden number} 
$$\eta := \frac{\sqrt{5} + 1}{2}\,\, .$$ 
Our construction (Theorem (\ref{exsitence})) will be done in Section 4, based on a lattice-theoretic preparations in Section 3. See also Remarks (\ref{nonmimimal}), (\ref{high}) for relevant results. After posting this note on ArXiv, Professor Bert van Geemen (\cite{Ge}) informed me a very impressive coincidence that the surfaces and also the free automorphisms in our note have been described by Cayley (\cite{Cay}) in 1870 in completely different ways, and that a paper explaining this is in preparation.

{\bf Acknowledgement.} I would like to express my thank to Professor Shu Kawaguchi for valuable discussion and to Professors Mitsuhiro Shishikura and Hiroki Sumi for their encouragement and interest in this work, and to Professor Bert van Geemen for his interest in this work and a very important imformation about Cayley's work.  

\section{smooth compact K\"ahler surface with a free automorphism of positive entropy}

In this section, we shall prove Theorem (\ref{main}) (1). 

Let $M$ be a compact K\"ahler manifold of dimension $n$ 
and $g \in {\rm Aut}\, (M)$ 
be a biholomorphic automorphism of $M$. We define:
$$T(M, g) := \sum_{k=0}^{2n} (-1)^k {\rm tr}\, g^* \vert H^k(M, {\mathbf Z})\,\, ;$$
$$H(M, g) := \sum_{k=0}^{n} (-1)^k {\rm tr}\, g^* \vert H^k(M, {\mathcal O}_M)\,\, .$$
Here ${\rm tr}\, f$ is the trace of $f$. $T(M, g)$ is the {\it topological Lefschetz number} and $H(M, g)$ is the 
{\it holomorphic Lefschetz number}. It is not hard to see that 
$T(M, g)$ (resp. $H(M, g)$) is the integration of the Poincar\'e dual $\sigma_{\Delta}$ of the diagonal $\Delta \subset M \times M$ (resp. the sum of $(0, *)$-th Kunneth components of $\sigma_{\Delta}$) over the graph $\Gamma_{g}$ of $g$ (see eg. \cite{GH}, Pages 420-421, 423-424). Therefore, without more precise formula, we have:
\begin{proposition}\label{lefschetz} If $M^g = \emptyset$, then 
$T(M, g) = L(M, g) = 0$. 
\end{proposition}
Let $S$ be a smooth compact K\"ahler surface and $g \in {\rm Aut}\, (S)$ 
be a biholomorphic automorphism of $S$. 
\begin{lemma}\label{rational} Assume that $S$ is bimeromorphic (in fact, birational) to either ${\mathbf P}^2$ or an Enriques surface. Then $S^g \not= \emptyset$.
\end{lemma}
\begin{proof} Note that $H^k(S, {\mathcal O}_S) = 0$ for $k \ge 1$. Thus
$$H(S, g) = {\rm tr}\, (g^* \vert H^0(M, {\mathcal O}_S)) = 1 \not= 0\,\, .$$
Hence $S^g \not= \emptyset$. 
\end{proof}
\begin{lemma}\label{others} Assume that $S$ is bimeromorphic to either $2$-dimensional complex torus or a K3 surface. Let $\overline{S}$ be the minimal model 
of $S$ and $\pi : S \rightarrow \overline{S}$ be the naturally 
induced morphism. Then $g$ descends to an automorphism $\overline{g}$ of $\overline{S}$. Moreover, $\overline{g}$ is of positive entropy (resp. free) if and only if $g$ is of positive entropy (resp. free). 
\end{lemma}
\begin{proof} The first assertion follows from the uniqueness of the minimal model of $S$ with non-negative Kodaira dimension (see eg. \cite{BHPV}, Page 99, Proposition (4.6)). It is clear that $g$ is free if so is $\overline{g}$. Conversely, if $\overline{g}$ has a fixed point, say $P$, then $g$ acts on $\pi^{-1}(P)$, which is either a point, say $Q$, or a tree of ${\mathbf P}^1$. In the first case $g(Q) = Q$. In the second case, let $C$ be the proper transform of the exceptional curve at the first blow up at $P$. Then $g(C) = C$ and $g$ has at least one fixed point (as $C \simeq {\mathbf P}^1$). Hence $\overline{g}$ is free if and only if so is $g$. Consider the following equivariant decomposition: 
$$H^2(S, {\mathbf Z}) = \pi^*H^2(\overline{S}, {\mathbf Z}) \oplus E$$
where $E$ is the sublattice generated by the classes of exceptional divisors. 
The lattice $E$ is of negative definite and $g^{*}(E) = E$. So, the eigenvalues $g^{*} \vert E$ are of absolute value $1$. Hence $d_1(g) >1$ if and only if $d_1(\overline{g}) > 1$. 
\end{proof}
\begin{lemma}\label{torus} Assume that $S$ is bimeromorphic to a $2$-dimensional complex torus. Assume further that $S$ has a free automorphism $g$. 
Then $g$ is of null entropy. 
\end{lemma}
\begin{proof} By Lemma (\ref{others}), we may (and will) assume that $S$ itself is minimal, i.e., $S$ is a $2$-dimensional complex torus. Then under the global linear coordinates ${\mathbf x} := (x_1, x_2)$ of the universal cover ${\mathbf C}^2$ of $S$, the automorphism $g$ is written as in the following form:
$$g^*({\mathbf x}) = A{\mathbf x} + {\mathbf b}\,\, ,$$
where $A \in {\rm GL}_2({\mathbf C})$ and ${\mathbf b} \in {\mathbf C}^2$. 
Let $\alpha$ and $\beta$ be the eigenvalues of $A$ 
counted with multiplicities. Then $\alpha$ and $\beta$ are the eigenvalues of the action of $g$ on $H^0(S, \Omega_S^1)$. By the Hodge duality together with the fact that $H^0(S, \Omega_S^2) = \wedge^2 H^0(S, \Omega_S^1)$, 
we obtain
$$H(S, g) = 1 - (\overline{\alpha} + \overline{\beta}) + \overline{\alpha}\overline{\beta} = (1 -\overline{\alpha})(1-\overline{\beta})\,\, .$$
Since $H(S, g) = 0$, it follows that either $\alpha = 1$ or $\beta = 1$. 
Moreover, by 
$$H^1(S, {\mathbf Z}) \otimes {\mathbf C} = H^0(S, \Omega_S^1) \oplus \overline{H^0(S, \Omega_S^1)}\,\, ,$$
it follows that the eigenvalues of (the ${\mathbf C}$-linear extension of) $g^* \vert H^1(S, {\mathbf Z})$ are $\alpha$, $\beta$, $\overline{\alpha}$ and $\overline{\beta}$ counted with multiplicities. Thus 
$$\alpha\beta\overline{\alpha}\overline{\beta} = \pm 1\,\, ,$$
because $g^* \vert H^1(S, {\mathbf Z}) \in {\rm GL}(H^1(S, {\mathbf Z})) \simeq {\rm GL}_4({\mathbf Z})$. 
Hence if either $\alpha = 1$ or $\beta = 1$, then both $\vert \alpha \vert = 1$ and $\vert \beta \vert = 1$ hold. Since $H^2(S, {\mathbf Z}) = \wedge^2 H^1(S, {\mathbf Z})$, it follows that the eigenvalues of $g^* \vert H^2(S, {\mathbf Z})$ are of absolute value $1$. Hence $d_1(g) = 1$. 
\end{proof}
\begin{lemma}\label{k3} Assume that $S$ is bimeromorphic to a K3 surface 
$\overline{S}$. 
Let $\omega_S$ be a generator of $H^0(\Omega_S^2) \simeq {\mathbf C}$. 
Assume further that $S$ has a free automorphism $g$. 
Then $g^* \omega_S = -\omega_S$ and $S$ is projective. Moreover 
$\rho(\overline{S}) \ge 2$ where $\rho(\overline{S})$ is the Picard number 
of $\overline{S}$.
\end{lemma}
\begin{proof} We use the same notation as in Lemma (\ref{others}). 
Note that $H^1({\mathcal O}_S) = 0$ and that $H^0(\Omega_S^2) = 
{\mathbf C}\omega_S$ is the Serre dual of $H^2({\mathcal O}_S)$. 
Let $g^* \omega_S = \alpha \omega_S$. Then 
$$0 = H(S, g) = 1 + \alpha^{-1}\,\, $$ 
and therefore $\alpha = -1$. This also implies that $\overline{g}^*\omega_{\overline{S}} = -\omega_{\overline{S}}$. Therefore $\overline{S}$ is projective 
by Theorem (\ref{hk}) below. 
Since the induced automorphism $\overline{g}$ of $\overline{S}$ is free by Lemma (\ref{others}),  it follows that 
$T(\overline{S}, \overline{g}) = 0$. Let ${\rm NS}\,(\overline{S})$ 
be the N\'eron-Severi lattice of $\overline{S}$ and $T(\overline{S})$ 
be the transcendental lattice of $\overline{S}$, i.e., 
$$T(\overline{S}) = \{ x \in H^2(\overline{S}, {\mathbf Z}) \vert (x.\, {\rm NS}\, (\overline{S})) = 0\}\,\, .$$ 
Since $\overline{S}$ is projective, it follows that the lattice ${\rm NS}\, (\overline{S})$ is non-degenerate and therefore ${\rm NS}\, (\overline{S}) \oplus T(\overline{S})$ is a subgroup of finite index of $H^2(\overline{S}, {\mathbf Z})$. We also note that ${\rm rank}\, {\rm NS}\, (\overline{S}) = \rho(\overline{S}) \ge 1$ (as $\overline{S}$ is projective) and ${\rm rank}\, T(\overline{S}) = 22 - \rho(\overline{S})$. 
Thus 
$$0 = T(\overline{S}, \overline{g}) = 2 + {\rm tr}\, \overline{g}^* \vert {\rm NS}\, (\overline{S}) + {\rm tr}\, \overline{g}^* \vert T(\overline{S})\,\, .$$
Since $\overline{g}^* \omega_{\overline{S}} = -\omega_{\overline{S}}$, it follows that $\overline{g}^* \vert T(\overline{S}) = -id_{T(\overline{S})}$ 
by \cite{Ni1}, Theorem 0.1 (see also \cite{Og}, Theorem 2.4 (2)). So, if $\rho(\overline{S}) = 1$, 
then 
$$T(\overline{S}, \overline{g}) = 2 + 1 - (22 -1) = -20 \not= 0\,\, .$$
Therefore $\rho(\overline{S}) \not= 1$, whence $\rho(\overline{S}) \ge 2$.   
\end{proof}
\begin{theorem}\label{hk} Let $X$ be a compact hyperk\"ahler manifold, that is, a smooth simply-connected compact K\"ahler manifold with an everywhere nondegenerate holomorphic $2$-form $\omega_X$ such that $H^0(\Omega^2_X) 
= {\mathbf C}\omega_X$. Assume that $X$ admits a bimeromorphic 
automorphism $g$ such that $g^*\omega_X = \zeta_n \omega_X$, where 
$\zeta_n \not= 1$ and $\zeta_n$ is a root of unity. Then $X$ is 
projective.
\end{theorem}
\begin{proof} This theorem is first proved by Nikulin (\cite{Ni1}, Theorem (0.1)) when $X$ is a K3 surface and $g$ is of finite order. Then this theorem is generalized by Beauville (\cite{Be}, Proposition 6, see also Proposition 7) when $X$ is a hyperk\"ahler manifold and $g$ is a biholomorphic automorphism of finite order and by \cite{Og}, Theorem (2.4) when $X$ is a hyperk\"ahler manifold and $g$ is a bimeromorphic automorphism (not necessarily of finite order). 
\end{proof} 
The following remark was the starting point of our construction in sections 3 and 4, even though it is no longer needed in the construction. 
\begin{remark}\label{neron-severi} Let$S$ be a K3 surface of Picard number $2$ 
with a free automorphism $g$. Let $\varphi(t)$ be the characteristic polynomial of $g^* \vert {\rm NS}\, (S)$. Then $\varphi(t)$ is either 
$t^2 - 18t + 1$ or $t^2 - 18t - 1$. In particular, $g$ is of positive entropy. 
Moreover, if $\varphi(t) = t^2 - 18t + 1$, then the eigenvalues 
of $g^* \vert {\rm NS}\, (S)$ are $9 \pm 4\sqrt{5} = \eta^{\pm 6}$, where $\eta$ is the golden number (see Introduction for the definition of $\eta$). 
\end{remark}
\begin{proof} Since $\rho(S) =2$ and $g^* \vert T(S) = -id_{T(S)}$ as observed in the proof of Lemma (\ref{k3}), it follows that
$$0 = H(S, g) = 2 + {\rm tr}\, g^* \vert {\rm NS}\, (S) + (-1)(22 - 2)\,\, .$$
Hence ${\rm tr}\, g^* \vert {\rm NS}\, (S) = 18$. Since $g^* \vert {\rm NS}\, (S) \in {\rm GL}\, ({\rm NS}\, (S)) \simeq {\rm GL}_2({\mathbf Z})$, 
it follows that ${\rm det}\, (g^* \vert {\rm NS}\, (S)) = \pm 1$. 
Thus $\varphi(t) = t^2 - 18t \pm 1$, and one of eigenvalues of $g^* \vert {\rm NS}\, (S)$ is greater than $1$. The last assertion can be checked by a straightforward calculation.
\end{proof}

Now we are ready to complete the proof of Theorem (\ref{main}) (1). 
Let $S$ be a smooth compact K\"ahler surface and $g$ be a free automorphism of positive entropy. Since $g$ is of positive entropy, $S$ 
is bimeromorphic to either ${\mathbf P}^2$, an Enriques surface, $2$-dimensional complex torus or a K3 surface by Theorem (\ref{cantat}). Since $g$ is free, the first two cases are ruled out by Lemma (\ref{rational}). Since $g$ is free but of positive entropy, the third case is also ruled out by Lemma (\ref{torus}). Hence 
$S$ is birational to a projective K3 surface of Picard number $\ge 2$ by Lemma (\ref{k3}). This proves Theorem (\ref{main})(1).

\section{hyperbolic lattices arizing from the golden number}

In this section, we shall construct even hyperbolic lattices from the golden number $\eta$ (see Introduction for the definition of $\eta$).
The minimal polynomial of $\eta$ over $\mathbf Z$ is $t^2 - t -1$,  
$$N := {\mathbf Z}[\eta]$$ 
is the ring of algebraic integers of ${\mathbf Q}(\sqrt{5})$ and $\langle \eta \rangle \times \langle -1 \rangle$ is the unit group ${\mathbf Z}[\eta]^{\times}$ of $N$. 
\begin{lemma}\label{fibonacci} Let $\{a_n\}_{n \ge 0}$ be the 
Fibonacci sequence:
$$a_{n+2} = a_{n+1} + a_n\,\, ,\,\, a_0 = 0\,\, ,\,\, a_1 = 1\,\, .$$
Then $\eta^n = a_n \eta + a_{n-1}$ for each positive integer $n$. For instance 
$\eta^3 = 2\eta + 1$, $\eta^4 = 3\eta + 2$, $\eta^5 = 5\eta + 3$, $\eta^6 = 8 \eta + 5$, $\eta^7 = 13 \eta + 8$. 
\end{lemma}
\begin{proof} Argue by induction on $n$ using the fact that $\eta^2 = \eta + 1$.
\end{proof}

The elements $e_1 := 1$ and $e_2 := \eta$ form the free basis of $N$ as ${\mathbf Z}$-module:
$$N = {\mathbf Z}e_1 \oplus {\mathbf Z}e_2 \simeq {\mathbf Z}^2\,\, .$$
Let $n \in {\mathbf Z}$. Since $\eta \in {\mathbf Z}[\eta]^{\times}$, the homomorphism
$$\eta^n : N \rightarrow N\,\, ;\,\, p(\eta) \mapsto \eta^np(\eta)$$
is an automorphism of the ${\mathbf Z}$-module $N$. 
\begin{lemma}\label{ev} Let $n$ be a positive integer. 

(1) The eigenvalues of the automorphism $\eta^{2n}$ on $N$ 
are $\eta^{2n}$ and $1/\eta^{2n}$. 

(2) The characteristic polynomial of $\eta^{2n}$ on $N$ is $t^2 - (a_{2n} + 2a_{2n-1}) +1$, where $\{a_n\}_{n \ge 0}$ is the Fibonacci sequence. For instance, the characteristic polynomials of $\eta^2$, $\eta^4$, $\eta^6$ are $t^2 -3t +1$, $t^2 - 7t +1$, $t^2 - 18t +1$ respectively. 
\end{lemma}
\begin{proof} The eigenvalues of $\eta$ on $N$ 
are 
$$\frac{1 \pm \sqrt{5}}{2} = \eta\,\, ,\,\, \frac{-1}{\eta}\,\, .$$
Hence the assertion (1) follows. By Lemma (\ref{fibonacci}), we have 
$$\eta^{2n} = a_{2n} \eta + a_{2n-1}\,\, .$$
Taking the Galois conjugate $\eta \mapsto -1/\eta$, we obtain 
$$\frac{1}{\eta^{2n}} = a_{2n} \frac{-1}{\eta} + a_{2n-1}\,\, .$$
Thus 
$$\eta^{2n} + \frac{1}{\eta^{2n}} = a_{2n} (\eta + \frac{-1}{\eta}) + 2a_{2n-1} = a_{2n} + 2a_{2n-1}\,\, .$$
This together with (1) implies (2).
\end{proof} 
Let 
$$b : N \times N \rightarrow {\mathbf Z}$$
be a ${\mathbf Z}$-valued symmetric bilinear form on $N$. With respect to the basis $\langle e_1, e_2 \rangle$, the form $b$ is represented by the symmetric matrix 
$$S_b := \left(\begin{array}{rr}
p & q\\
q & r\\
\end{array} \right)\,\, ,$$
where $p = b(e_1, e_1)$, $q = b(e_1, e_2)$ and $r = b(e_2, e_2)$. In terms of the matrix $S_b$, the form $b$ even if $p$, $r$ are even integers and $b$ is hyperbolic if the matrix $S_b$ is of signature $(1,1)$. We call a pair $(N, b)$ an 
{\it even hyperbolic lattice} if $b$ is a ${\mathbf Z}$-valued even hyperbolic symmetric bilinear form on $N$. We have then a natural embedding 
$$N \subset N^* := {\rm Hom}_{{\mathbf Z}}(N, {\mathbf Z}) \subset N \otimes {\mathbf Q}$$ 
by $b$. The quotient group $N^*/N$ is called the {\it discriminant group} 
of $N$. 

An automorphism $f : N \rightarrow N$ as ${\mathbf Z}$-module is called an 
{\it isometry} of the lattice $(N, b)$ if $b(f(x), f(y)) = b(x, y)$ for all $x, y \in N$. Note that an isometry $f$ naturally induces an automorphism (as abelian group) of the discriminat group $N^*/N$. 

We are interested in an even hyperbolic lattice structure $(N, b)$ on our $N$ such that $\eta^{2}$, whence $\eta^{2n}$ for all integer $n$, become isometries.

\begin{proposition}\label{lattice} Assume that $(N, b)$ is an even hyperbolic lattice and that $\eta^2$ is an isometry of $(N, b)$. Then: 

(1) The matrix 
$S_b$ is of the following form:
$$S_b := \left(\begin{array}{rr}
2q & q\\
q & -2q\\
\end{array} \right)\,\, ,$$
where $q$ is a non-zero integer (and vice versa). 

(2) Under (1), the discriminant group $N^*/N$ satisfies
$$N^*/N = \langle \frac{e_2}{q} \rangle \oplus \langle \frac{e_1 - 2e_2}{5q} \rangle \simeq {\mathbf Z}/q \oplus {\mathbf Z}/5q\,\, .$$ 

(3) Under (1), $b$ does not represent $0$, i.e., there is no $x \in N$ 
such that $b(x, x) = 0$. Moreover $b$ does not 
represent $\pm 2$, i.e., there is no $x \in N$ 
such that $b(x, x) = \pm 2$ if and only if $q \not= \pm 1$. 

(4) Under (1), $\eta^6$ acts on the discriminant group $N^*/N$ 
as $-id_{N^*/N}$ if and only if 
$q$ is one of $\pm 1$, $\pm 2$.
\end{proposition}
\begin{proof} Let $2p = b(e_1, e_1)$, $q = b(e_1, e_2)$ and 
$2r = b(e_2, e_2)$. 
By Lemma (\ref{fibonacci}), 
$$\eta^2(e_1) = \eta^2 = \eta + 1 = e_1 + e_2\,\, ;\,\, \eta^2(e_2) = \eta^3 = 2\eta + 1 = e_1 + 2e_2\,\, .$$
Hence $\eta^2$ is an isometry if and only if 
$$b(e_1, e_1) = b(e_1 + e_2, e_1 + e_2)\,\, ,\,\, b(e_1, e_2) = b(e_1 + e_2, e_1 + 2e_2)\,\, ,\,\, b(e_2, e_2) = b(e_1 + 2e_2, e_1 + 2e_2)\,\, ,$$
that is,
$$2p = 2p + 2q + 2r\,\, ,\,\, q = 2p + 3q + 4r\,\, ,\,\, 2r = 2p + 4q + 8r\,\, .$$
This is equivalent to $r = -q$ and $p = q$. Now it is straightforward to see that $S_b$ is hyperbolic if and only if $q \not= 0$. This proves (1). 

By using (1), we compute that 
$$b(e_1, \frac{e_2}{q}) = 1\,\, ,\,\, b(e_2, \frac{e_2}{q}) = -2\,\, ;\,\, 
b(e_1, \frac{e_1 - 2e_2}{5q}) = 0\,\, ,\,\, b(e_2, \frac{e_1 - 2e_2}{5q}) = 1\,\, .$$
Hence $e_2/q, (e_1 - 2e_2)/5q \in N^*/N$. Looking at the shape, we see that $e_2/q, (e_1 - 2e_2)/5q$ are of order $\vert q \vert$ and $\vert 5q \vert$ respectively and $\langle e_2/q \rangle \cap \langle (e_1 - 2e_2)/5q \rangle = 0$ in $N^*/N$. Combining this with the fact that
$$\vert N^*/N \vert = \vert {\rm det}\, S_ b \vert = 5q^2\,\, ,$$
we obtain (2).   

Write $x = ae_1 +be_2$ ($a, b \in {\mathbf Z}$). Then 
$$b(x, x) = 2q(a^2 + ab - b^2)\,\, .$$
The result (3) follows from this equality. 

Let us prove (4). By Lemma (\ref{fibonacci}), we have $\eta^6 = 8 \eta + 5$. $\eta^7 = 13 \eta + 8$. That is,  
$$\eta^6 (e_1) = 5e_1 + 8e_2\,\, ,\,\, \eta^6(e_2) = 8e_1+ 13e_2\,\, .$$
Thus 
$$\eta^6(\frac{e_2}{q}) = \frac{8e_1+ 13e_2}{q}\,\, ,\,\, 
\eta^6(\frac{e_1 - 2e_2}{5q}) = -\frac{11e_1+ 18e_2}{5q}\,\, .$$
Hence $\eta^6 \vert N^*/N = -id_{N^*/N}$ if and only if both
$$\frac{8e_1+ 14e_2}{q}\,\, ,\,\, \frac{2e_1+ 4e_2}{q}$$
are in $N$. This is equivalent to $q \vert 2$. This proves (4). 
\end{proof}
Taking $q = 2$ in Proposition (\ref{lattice}), we obtain the following result 
(see also Lemma (\ref{ev}) for the first statement of (2)):
\begin{corollary}\label{desiredlattice}
Let $(N, b)$ be the lattice (on our $N$) given by the matrix:
$$S_b := \left(\begin{array}{rr}
4 & 2\\
2 & -4\\
\end{array} \right)\,\, ,$$
i.e., the case where $q = 2$ in Proposition (\ref{lattice}). Then: 

(1) $(N, b)$ is an even hyperbolic lattice which represents neither 
$0$ nor $\pm 2$.

(2) $\eta^6$ is an isometry of $(N, b)$ such that the characteristic 
polynomial 
is $t^2 -18t +1$ and the induced action on the discriminant group 
$N^*/N$ is $-id_{N^*/N}$. 
\end{corollary}

\section{K3 surfaces with a free automorphism of positive entropy}

In this section, we shall prove Theorem (\ref{main}) (2) by 
showing the following more explicit:

\begin{theorem}\label{exsitence} There exists a projective K3 surface 
$S$ of $\rho(S) = 2$ such that 
${\rm NS}\, (S) = {\mathbf Z}h_1 \oplus {\mathbf Z}h_2$ where 
$$((h_i.h_j)) = 
\left(\begin{array}{rr}
4 & 2\\
2 & -4\\
\end{array} \right)\,\, .$$
Any such K3 surface $S$ (all of which form then a dense subset of $18$-dimensional family) admits a free automorphism $g$ of positive entropy. 
\end{theorem}

\begin{proof} Note that the abstract lattice given by the symmetric matrix above is an even hyperbolic lattice of rank $2$. Hence the first result follows from \cite{Mo}, Corollary (2.9), which is based on the surjectivity of the period map for K3 surfaces (see eg. \cite{BHPV}, Page 338, Theorem 14.1) and Nikulin's theory (\cite{Ni2}) of integral bilinear form. 

Let us construct an automorphism $g$ of $S$ with desired properties. 

Note that there is a lattice isomorphism 
$\varphi : {\rm NS}\, (S) \simeq (N, b)$,  where $(N, b)$ 
is the lattice in Corollary (\ref{desiredlattice}). 

Let $f := \varphi^{-1} \circ \eta^6 \circ \varphi$. Then $f$ is an isometry 
of ${\rm NS}\, (S)$. The eigenvalues of $\eta^6$ on $N$ 
are $\eta^6$ and $\eta^{-6}$ by Lemma (\ref{ev}). Thus so are the eigenvalues 
of $f$. Since both eigenvalues are positive, it follows that $f$ preserves the positive cone, i.e., the connected component of 
$$\{x \in {\rm NS}\, (S) \otimes {\mathbf R} \vert (x, x) > 0\}$$ 
that contains ample classes. Since ${\rm NS}\, (S) \simeq (N, b)$ does not represent $-2$, the ample cone of $S$ coincides with the positive cone. Thus $f$ preserves the ample cone. 

Since $\eta^6$ acts on the discriminant group $N^*/N$ as $-id_{N^*/N}$, 
the isometry $f$ also acts on the discriminant group ${\rm NS}\, (S)^*/NS(S)$ 
as $-id_{{\rm NS}\, (S)^*/{\rm NS}\, (S)}$. 

Let $T(S)$ be the transcendental lattice of $S$. Then $-id_{T(S)}$ also 
acts on the discriminant group $T(S)^*/T(S)$ as $-id_{T(S)^*/T(S)}$. 

Hence the isometry $(f, -id_{T(S)})$ of ${\rm NS}\, (S) \oplus T(S)$ extends 
to an isometry, say $\tilde{f}$, of $H^2(S, {\mathbf Z})$, by \cite{Ni2}, Proposition (1.6.1). Here we note that $H^2(S, {\mathbf Z})$ is unimodular and 
${\rm NS}\, (S)$ and $T(S) = {\rm NS}\, (S)^{\perp}$ (in $H^2(S, {\mathbf Z})$) are both primitive in $H^2(S, {\mathbf Z})$. By construction, $\tilde{f}$ preserves the Hodge decomposition of $H^2(S, {\mathbf Z})$ and, as observed above, also preserves the ample cone. Hence, by the golobal Torelli theorem for K3 surfaces (see eg. \cite{BHPV}, Chapter VIII Theorem (11.1)), there is an automorphism $g$ of $S$ such that $g^* \vert H^2(S, {\mathbf Z}) = \tilde{f}$. 

Since $f = \tilde{f} \vert {\rm NS}\,(S)$ has an eigenvalue $\eta^6 > 1$, it follows that $g$ is of positive entropy. 

Let us show that $g$ is free. Since $g \not= id_S$, the fixed point set $S^g$ consists of at most finitely many complete curves and at most 
finitely many points.

We first show that $S^g$ contains no complete curve. In fact, if $g(C) = C$ for some complete curve $C$, then the class $[C]$ would be an eigenvector of $g^* \vert {\rm NS}\, (S)$ with eigenvalue $1$. However, the eigenvalues of $g^* \vert {\rm NS}\, (S) = f$ are $\eta^6$ and $\eta^{-6}$, none of which is $1$. 

Hence $S^g$ consists of at most finitely many points, say $n$ points counted with multiplicities. Then, by the topological Lefschetz fixed point formula, we have
$$n = T(S, g) = 2 + {\rm tr}\, g^* \vert {\rm NS}\, (S) +  {\rm tr}\, g^* \vert T(S)\,\, .$$
By $g^* \vert {\rm NS}\, (S) = f$ and by 
${\rm tr}\, f = {\rm tr}\, (\eta^6 \vert N) = 18$ (by Corollary (\ref{desiredlattice}) (2)), we obtain 
$${\rm tr}\, g^* \vert {\rm NS}\, (S) = 18\,\,.$$
On the other hand, by $g^* \vert T(S) = -id_{T(S)}$ and 
${\rm rank}\, T(S) = 20$ (by $\rho(S) = 2$), we obtain
$${\rm tr}\, g^* \vert T(S) = -20\,\,.$$
Hence 
$$n = T(S, g) = 2 + 18 -20 = 0\,\, .$$
Hence $S^g = \emptyset$. This completes the proof. 
\end{proof}
\begin{remark}\label{eq} Let $(S, g)$ be as in Theorem (\ref{exsitence}). By the shape of ${\rm NS}\, (S)$ and the fact that ${\rm NS}\, (S)$ represents neither $0$ nor $\pm 2$, it follows from \cite{SD}, Theorem (6.1) that $S$ is realized as a smooth quartic surface in ${\mathbf P}^3$. It seems extremely hard but highly interesting to write down explcitly the equation of $S$ and the action 
of $g$ in terms of the global homogeneous coordinates of ${\mathbf P}^3$, for at least one of such pairs. 
\end{remark}
\begin{remark}\label{nonmimimal} Let $(S, g)$ be as in 
Theorem (\ref{exsitence}). Then 
$$T(S, g^2) = 2 + \eta^{12} + \frac{1}{\eta^{12}} + 20 > 0\,\, .$$ 
Therefore $g^2$ has a fixed point, say $P$. Put $Q = g(P)$. Then $Q \not= P$ 
(by $S^g = \emptyset$) and $g(\{P, Q\}) = \{P, Q\}$. Let $\tilde{S}$ 
be a smooth surface obtained by the blow up of $S$ at $P$ and $Q$. Then $g$ lifts to a free automorphism of $\tilde{S}$ with positive entropy. So, 
there is also a ``non-mimimal" K3 surface with a free automorphism of positive entropy.  
\end{remark}
\begin{remark}\label{high} For each integer $n \ge 2$, there is an $n$-dimensional smooth compact K\"ahler manifold wih a free automorphism of positive 
entropy. One of the ``cheapest" way to construct may be as follows. 
Let $(S, g)$ be as in 
Theorem (\ref{exsitence}) and ${\mathbf P}^{n-2}$ be the projectiv space of dimension $n-2$. Then the pair of product type
$$(S \times {\mathbf P}^{n-2}, g \times id_{{\mathbf P}^{n-2}})$$ 
satisfies the desired property. Needless to say, it would be more interesting to see if it is possible to construct free automorphisms of positive entropy ``which do not come from lower dimensional pieces" for higher dimensional manifolds.
\end{remark}

\end{document}